\documentclass[twoside,11pt]{article}
\usepackage{alex}
\usepackage{jmlr2e}
\usepackage{graphicx}
\usepackage{algorithm}
\usepackage{algorithmic}
\usepackage{alex}
\usepackage{natbib}
\usepackage[usenames]{color}
\definecolor{Red}{rgb}{0.9,0.1,0.1}

\newcommand{\tightsum}{  \hspace{-5mm}\sum_{j=1}^{\min(\tau, t-(\tau+1))}\hspace{-5mm}}

\jmlrheading{1}{2099}{1-99}{01/09}{01/09}{John Langford, Alex Smola,
  and Martin Zinkevich}

\ShortHeadings{Slow Learners are Fast}{Langford, Smola, and Zinkevich}

\firstpageno{1}

\begin{document}
\title{Slow Learners are Fast}

\author{%
  \name John Langford  \email jl@hunch.net \\
  \name Alexander J. Smola \email alex@smola.org \\
  \name Martin Zinkevich \email maz@yahoo-inc.com \\
  \addr Yahoo!~Labs, Great America Parkway, Santa Clara, CA 95051 USA
}

\editor{U.N. Known --- to be submitted to JMLR}

\maketitle

\begin{abstract}%
  Online learning algorithms have impressive convergence properties
  when it comes to risk minimization and convex games on very large
  problems. However, they are inherently sequential in their design
  which prevents them from taking advantage of modern multi-core
  architectures. In this paper we prove that online learning with
  delayed updates converges well, thereby facilitating parallel online
  learning.
\end{abstract}



\section{Introduction}
\label{sec:intro}

Online learning has become the paradigm of choice for tackling very
large scale estimation problems. Their convergence properties are well
understood and have been analyzed in a number of different frameworks
such as by means of asymptotics \cite[]{MurYosAma94}, game theory
\cite[]{HazAgaSat07}, or stochastic programming
\cite[]{NesVia00}. Moreover, learning-theory guarantees show that $O(1)$
passes over a dataset suffice to obtain optimal estimates
\cite[]{BotLeC04,BotBou07}. All those properties combined suggest that online
algorithms are an excellent tool for addressing learning problems.

This view, however, is slightly deceptive for several reasons: current
online algorithms process one instance at a time. That is, they
receive the instance, make some prediction, incur a loss, and update
an associated parameter. In other words, the algorithms are entirely
sequential in their nature. While this is acceptable in single-core
processors, it is highly undesirable given that the number of
processing elements available to an algorithm is growing exponentially
(e.g.\ modern desktop machines have up to 8 cores, graphics cards up
to 1024 cores). It is therefore very wasteful if only one of these
cores is actually used for estimation.

A second problem arises from the fact that network and disk I/O have not
been able to keep up with the increase in processor speed. A typical
network interface has a throughput of 100MB/s and disk arrays have
comparable parameters. This means that current algorithms reach their
limit at problems of size 1TB whenever the algorithm is I/O bound
(this amounts to a training time of 3 hours), or even smaller problems
whenever the model parametrization makes the algorithm CPU
bound.

Finally, distributed and cloud computing are unsuitable for today's
online learning algorithms. This creates a
pressing need to design algorithms which break the sequential
bottleneck. We propose two variants. To our
knowledge, this is the first paper which provides \emph{theoretical
  guarantees} combined with empirical evidence for such an
algorithm. Previous work, e.g.\ by \cite{DelBen07} proved rather
inconclusive in terms of theoretical and empirical guarantees.

In a nutshell, we propose the following two variants: several
processing cores perform stochastic gradient descent
\emph{independently} of each other while sharing a common parameter
vector which is updated asynchronously. This allows us to accelerate
computationally intensive problems whenever gradient
computations are relatively expensive. A second variant assumes that
we have linear function classes where parts of the function can be
computed \emph{independently} on several cores. Subsequently the
results are combined and the combination is then used for a descent step.

A common feature of both algorithms is that the update occurs
with some delay: in the first case other
cores may have updated the parameter vector in the meantime, in the second case, other
cores may have already computed parts of the function for the
subsequent examples before an update.

\section{Algorithm}

\subsection{Platforms}

We begin with an overview of three platforms which are available for
parallelization of algorithms.
The differ in their structural parameters, such as
synchronization ability, latency, and bandwidth and consequently they
are better suited to different styles of algorithms. This description
is not comprehensive by any means. For instance, there exist numerous
variants of communication paradigms for distributed and cloud
computing ranging from fully independent Folding@Home algorithms
\cite[]{ShiPan00} to sophisticated pipelines like the Drayad
architecture \cite[]{Isardetal07}.

\begin{description}
\item[Shared Memory Architectures:] The commercially available 4-16 core
  CPUs on servers and desktop computers fall into this category. They
  are general purpose processors which operate on a joint memory space
  where each of the processors can execute arbitrary pieces of code
  independently of other processors. Synchronization is easy via
  shared memory/interrupts/locks. The critical shared resource is memory
  bandwidth. This problem can be somewhat alleviated by exploiting
  affinity of processes to specific cores.

  A second example of a shared memory architecture are graphics
  cards. There the number of processing elements is vastly higher (512
  on high-end consumer graphics cards), although they tend to be
  bundled into groups of 8 cores (also referred to as multiprocessing
  elements), each of which can execute a given piece of code in a
  data-parallel fashion. An issue is that explicit synchronization
  between multiprocessing elements is difficult --- it requires
  computing kernels on the processing elements to complete. This means
  that an explicit synchronization mechanism may be undesirable since
  it comes at the expense of a large performance penalty or a
  significant increase in latency. Implicit synchronization via shared
  memory is still possible. Critical resources are availability of
  memory: consumer grade graphics cards have in the order of 512MB
  high speed RAM per chip. Communication between multiple chips is
  nontrivial.
\item[Clusters:] To increase I/O bandwidth one can combine several
  computers in a cluster using MPI or PVM as the underlying
  communications mechanism. A clear limit here is bandwidth
  constraints and latency for inter-computer communication. On Gigabit
  Ethernet the TCP/IP latency can be in the order of 100$\mu$s, the
  equivalent of $10^5$ clock cycles on a processor and network
  bandwidth tends to be a factor 100 slower than memory
  bandwdith. Infiniband is approximately one order of magnitude faster
  but it is rarely found in off-the-shelf server farms.
\item[Grid Computing:] Computational paradigms such as MapReduce
  \cite[]{ChuKimLinYuetal07} are well suited
  for the parallelization of batch-style algorithms
  \cite[]{TeoVisSmoLe09}. In comparison to cluster configurations
  communication and latency are further constrained. For instance,
  often individual processing elements are unable to communicate
  directly with other elements with disk / network storage being the
  only mechanism of inter-process data transfer. Moreover, the latency
  is significantly increased, typically in the order of seconds, due
  to the interleaving of Map and Reduce processing stages.
\end{description}
Of the above three platform types we will only consider the first
two since latency plays a critical role in the analysis of the class
of algorithms we propose. While we do not exclude the possibility of
devising parallel online algorithms suited to grid computing, we
believe that the family of algorithm proposed in this paper is
unsuitable and a significantly different synchronization paradigm
would need to be explored.

\subsection{Delayed Stochastic Gradient Descent}

Many learning problems can be written as convex minimization
problems. It is our goal to find some parameter vector $x$ (which is
drawn from some Banach space $\Xcal$ with associated norm
$\nbr{\cdot}$) such that the sum over convex functions $f_i: \Xcal \to
\RR$ takes on the smallest value possible. For instance, (penalized)
maximum likelihood estimation in exponential families with fully
observed data falls into this category, so do Support Vector Machines
and their structured variants. This also applies to distributed games
with a communications constraint within a team.

At the outset we make no special assumptions on the order or form of
the functions $f_i$. In particular, an adversary may choose to order
or generate them in response to our previous choices of $x$. In other
cases, the functions $f_i$ may be drawn from some distribution (e.g.\
whenever we deal with induced losses). It is our
goal to find a sequence of $x_i$ such that the cumulative loss $\sum_i
f_i(x_i)$ is minimized.
With some abuse of notation we identify the average empirical and expected
loss \emph{both} by $f^*$. This is possible, simply by redefining $p(f)$ to
be the uniform distribution over $F$. Denote by
\begin{align}
  f^*(x) := \frac{1}{|F|} \sum_{i} f_i(x)
  \text{ or }
  f^*(x) & := \Eb_{f \sim p(f)}[f(x)] \\
  \text{ and correspondingly }
  x^* & := \argmin_{x \in \Xcal} f^*(x)
\end{align}
the average risk. We assume that $x^*$ exists (convexity does
\emph{not} guarantee a bounded minimizer) and that it satisfies
$\nbr{x^*} \leq R$ (this is always achievable, simply by intersecting
$\Xcal$ with the unit-ball of radius $R$). We propose the following
algorithm:

\begin{algorithm}[h]
  \caption{Delayed Stochastic Gradient Descent \label{alg:delay}}
   \begin{algorithmic}
    \STATE {\bfseries Input:} Feasible space $X\subseteq \RR^n$, annealing schedule $\eta_t$ and delay $\tau \in
    \NN$
    \STATE Initialization: set $x_1 \ldots, x_\tau = 0$ and compute corresponding $g_t = \nabla f_t(x_t)$.
    \FOR{$t = \tau + 1$ {\bfseries to} $T + \tau$}
    \STATE Obtain $f_t$ and incur loss $f_t(x_t)$
    \STATE Compute $g_t := \nabla f_t(x_t)$
    \STATE Update $x_{t+1} = \argmin_{x\in X}\nbr{x-(x_t - {\color{Red} \eta_t g_{t-\tau}})}$ (Gradient Step and Projection)
    \ENDFOR
  \end{algorithmic}
\end{algorithm}
In this paper the annealing schedule will be either $\eta_t=\frac{1}{\sigma (t - \tau)}$ or
$\eta_t=\frac{\sigma}{\sqrt{t-\tau}}$. Often, $X=\RR^n$.
If we set $\tau = 0$, algorithm~\ref{alg:delay} becomes an entirely
standard stochastic gradient descent algorithm. The only difference
with delayed stochastic gradient descent is that we do \emph{not}
update the parameter vector $x_t$ with the current gradient $g_t$ but
rather with a delayed gradient $g_{t-\tau}$ that we computed $\tau$
steps previously. Later we will extend this simple stochastic gradient
descent model in two ways: firstly we will extend the updates to
implicit updates as they arise from the use of Bregman divergences
(see Section~\ref{sec:strong}), leading to variants such as parallel
exponentiated gradient descent. Secondly, we will modify bounds which
are dependent on strong convexity \cite[]{BarHazRak08,DoLeFoo09} to
obtain adaptive algorithms which can take advantage of well-behaved
optimization problems in practice.

\subsection{Templates}

\paragraph{Asynchronous Optimization}

Assume that we have $n$ processors which can process data
independently of each other, e.g.\ in a multicore platform, a graphics
card, or a cluster of workstations. Moreover, assume that computing
the gradient of $f_t(x)$ is at least $n$ times as
expensive\footnote{More fine-grained variants are possible where we
  write only parts of the parameter vector $x$ at a time, thereby
  requiring locks on only parts of $x$ by an updating processor. We
  omit details of such modifications as they are entirely technical
  and do not add to the key idea of the paper.}  as it is to update
$x$ (read, add, write). This occurs, for instance, in the case of
conditional random fields \cite[]{RatBagZin07,VisSchSchMur06}, in
planning \cite[]{RatBagZin06}, and in ranking \cite[]{WeiKarLeSmo08}.

The rationale for delayed updates can be seen in the following
setting: assume that we have $n$ cores performing stochastic
gradient descent on different instances $f_t$ while sharing one common
parameter vector $x$. If we allow each core in a round-robin fashion
to update $x$ one at a time then there will be a delay of $\tau = n-1$
between when we see $f_t$ and when we get to update $x_{t+\tau}$. The
delay arises since updates by different cores cannot happen
simultaneously. This setting is preferable whenever computation of
$f_t$ itself is time consuming.

\begin{figure}[htb]
  \centering
  \includegraphics[width=0.7\textwidth]{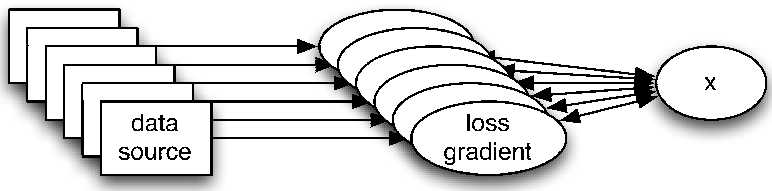}
  \caption{Data parallel stochastic gradient descent with shared
    parameter vector. Observations are partitioned on a per-instance
    basis among $n$ processing units. Each of them computes its own
    loss gradient $g_t = \partial_x f_t(x_t)$. Since each computer
    is updating $x$ in a round-robin fashion, it takes a delay of
    $\tau = n-1$ between gradient computation and when the gradients
    are applied to $x$.}
  \label{fig:roundrobin}
\end{figure}

Note that there is no need for explicit thread-level synchronization
between individual cores. All we need is a read / write-locking
mechanism for $x$ or alternatively, atomic updates on the parameter
vector.\footnote{There exists some limited support for this in the
  Intel Threading Building Blocks library for the x86 architecture.}
This is important since thread synchronization on GPUs is rudimentary
at best. Keeping the \emph{state} synchronized by a shared memory
architecture is key.


On a multi-computer cluster we can use a similar mechanism simply by
having one server act as a state-keeper which retains an up-to-date
copy of $x$ while the loss-gradient computation clients can retrieve
at any time a copy of $x$ and send gradient update messages to the
state keeper. Note that this is only feasible whenever the message
size does not exceed $\frac{1}{n}$ of the bandwidth of the
state-keeper. This suggests an alternative variant of the algorithm
which is considerably less demanding in terms of bandwidth
constraints.

\paragraph{Pipelined Optimization}

The key impediment in the previous template is that it required
significant amounts of bandwidth solely for the purpose of
synchronizing the state vector. This can be addressed by parallelizing
computing the function value $f_i(x)$ explicitly rather than
attempting to compute several instances of $f_i(x)$
simultaneously. Such situations occur, e.g.\ when $f_i(x) =
g(\inner{\phi(z_i)}{x})$ for high-dimensional $\phi(z_i)$. If we
decompose the data $z_i$ (or its features) over $n$ nodes we can
compute partial function values and also all partial updates
\emph{locally}. The only communication that is required is to combine
partial values and to compute the gradient with respect to
$\inner{\phi(z_i)}{x}$.

This causes delay since the second stage is processing results of the
first stage while the latter has already moved on to processing
$f_{t+1}$ or further. While the architecture is quite different, the
effects are identical: the parameter vector $x$ is updated with some
delay $\tau$. Note that here $\tau$ can be much smaller than the
number of processors and mainly depends on the latency of the
communication channel. Also note that in this configuration the memory
access for $x$ is entirely local.

\paragraph{Randomization}

Order of observations matters for delayed updates: imagine that an
adversary, aware of the delay $\tau$ bundles each
of the $\tau$ most similar instances $f_t$ together. In this case we
will incur a loss that can be $\tau$ times as large as in the
non-delayed case and require a learning rate which is $\tau$ times
smaller. The reason being that only after seeing $\tau$ instances of
$f_t$ will we be able to respond to the data. Such highly correlated
settings do occur in practice: for instance, e-mails or search
keywords have significant temporal correlation (holidays, political
events, time of day) and cannot be treated as iid data.

A simple strategy can be used to alleviate this problem: decorrelate
observations by random permutations of the instances. The price we pay
for this modification is a delay in updating the model parameters
(there need not be any delay in the \emph{prediction} itself) since
obviously the range of decorrelation needs to exceed $\tau$
considerably.

\section{Lipschitz Continuous Losses}

We begin with a simple game theoretic analysis that only requires $f_t$ to be
\emph{convex} and where the subdifferentials are bounded $\nbr{\nabla
  f_t(x)} \leq L$ by some $L > 0$. Denote by $x^*$ the minimizer of
$f^*(x)$. It is our goal to bound the regret $R$ associated
with a sequence $X = \cbr{x_1, \ldots, x_T}$ of parameters
\begin{align}
  \label{eq:regret}
  R[X] := \sum_{t=1}^T f_t(x_t) - f_t(x^*).
\end{align}
Such bounds can then be converted into bounds on the expected
loss. See e.g.\ \citep{ShaSinSre07} for an example of a randomized
conversion. Since all $f_t$ are convex we can upper bound $R[X]$ via
\begin{align}
  \label{eq:convbound}
  R[X] \leq \sum_{t=1}^T \inner{\nabla f_t(x_t)}{x_t - x^*} =
  \sum_{t=1}^T \inner{g_t}{x_t - x^*}.
\end{align}
Next define a potential function measuring the distance between $x_t$
and $x^*$. In the more general analysis this will become a Bregman
divergence. We define $D(x \| x') := \frac{1}{2} \nbr{x - x'}^2$.
To prove regret bounds we need the following auxiliary lemma which
bounds the instantaneous risk at a given time:
\begin{lemma}
  \label{lem:auxlemma}
  For all $x^*$ and for all $t >\tau$, if $X=\RR^n$, the following expansion holds:
  \begin{align}
    \label{eq:auxbound}
    \inner{x_{t-\tau} - x^*}{g_{t-\tau}} =
    \frac{1}{2} \eta_t \nbr{g_{t-\tau}}^2 +
    \frac{D(x^*\|x_{t}) - D(x^*\|x_{t+1})}{\eta_t} +
    \sum_{j=1}^{\min(\tau,t-(\tau+1))} \hspace{-4mm}
    \eta_{t-j} \inner{g_{t-\tau-j}}{g_{t-\tau}}
  \end{align}
  Furthermore, \eq{eq:auxbound} holds as an \emph{upper bound} if $X\subset \RR^n$.
\end{lemma}
\begin{proof}
The divergence function allows us to decompose our progress via
\begin{align}
  \label{eq:gaussdecom}
  D(x^* \| x_{t+1}) - D(x^* \| x_{t})
  & =
  \frac{1}{2} \nbr{x^* - x_t + x_t - x_{t+1}}^2 -
  \frac{1}{2} \nbr{x^* - x_t}^2 \\
  & =
  \frac{1}{2} \nbr{x^* - x_t + \eta_t g_{t-\tau}}^2 -
  \frac{1}{2} \nbr{x^* - x_t}^2 \\
  & = \frac{1}{2} \eta_t^2 \nbr{g_{t-\tau}}^2 - \eta_t \inner{x_t
    - x^*}{g_{t-\tau}} \\
  & = \frac{1}{2} \eta_t^2 \nbr{g_{t-\tau}}^2 - \eta_t
  \inner{x_{t-\tau} - x^*}{g_{t-\tau}} - \eta_t \inner{x_{t} -
    x_{t-\tau}}{g_{t-\tau}}
  \label{eq:gaussdecom-intermediate}
\end{align}
We can now expand the inner product between delayed parameters
$\inner{x_{t} - x_{t-\tau}}{g_{t-\tau}}$ in terms of differences
between gradients. Here we need to distinguish the initialization: for
$\tau \leq t < 2\tau$ we only obtain differences between $t-\tau$
gradients, since the optimization protocol initializes $x_t = x_1$ for
all $t \leq \tau$. This yields
\begin{align}
  \inner{x_{t} - x_{t-\tau}}{g_{t-\tau}} =\hspace{-8mm}
  \sum_{j=1}^{\min(t-(\tau+1), \tau)} \hspace{-7mm} \inner{x_{t-(j-1)} -  x_{t-j}}{g_{t-\tau}} =
  -\hspace{-8mm}\sum_{j=1}^{\min(t-(\tau+1), \tau)} \hspace{-7mm}\eta_{t-j} \inner{g_{t-\tau-j}}{g_{t-\tau}}
\end{align}
Plugging the above into \eq{eq:gaussdecom-intermediate},
dividing both sides by $\eta_t$ and moving
$\inner{x_{t-\tau} - x^*}{g_{t-\tau}}$ to the LHS completes the proof.

To show that the inequality holds note that distances between vectors
can only decrease if we project onto convex sets. The argument
follows that of \cite{Zinkevich03}.
\end{proof}
%
%
Note that the decomposition \eq{eq:auxbound} is very similar to
standard regret decomposition bounds, such as \citep{Zinkevich03}. The
key difference is that we now have an additional term characterizing
the correlation between successive gradients which needs to be
bounded. In the worst case all we can do is bound
$\inner{g_{t-\tau-j}}{g_{t-\tau}} \leq L^2$, whenever the gradients are
highly correlated, which leads to the following theorem:
\begin{theorem}
  \label{th:delay}
  Suppose all the cost functions are Lipschitz continuous with a constant $L$ and $\max_{x,x'\in X}D(x\|x')\leq F^2$. Given $\eta_t=\frac{\sigma}{\sqrt{t-\tau}}$ for some constant $\sigma>0$,
  the regret of the delayed update algorithm is bounded by
  \begin{align}
    R[X] \leq \sigma L^2 \sqrt{T} + F^2 \frac{\sqrt{T}}{\sigma} +
    L^2 \frac{\sigma \tau^2}{2} + 2 L^2 \sigma \tau \sqrt{T}
  \end{align}
  and consequently for $\sigma^2 = \frac{F^2}{2 \tau L^2}$ and $T \geq
  \tau^2$ we obtain the bound
  \begin{align}
    \label{eq:optsigma-delay}
    R[X] \leq 4FL \sqrt{\tau T}
  \end{align}
\end{theorem}
\begin{proof}
Before we prove the claim, we briefly state a few useful
identities concerning sums.
\begin{align}
  \label{eq:p1sum}
  \sum_{i=1}^n i & = \frac{n (n+1)}{2} \\
  \label{eq:pneg12sum}
  \sum_{i=a}^b \frac{1}{2 \sqrt{i}} & \leq \int_{a-1}^b \frac{1}{2 \sqrt{x}} dx =
  \sqrt{b} - \sqrt{a-1} \leq \sqrt{b-a+1}
\end{align}
Summing over \eq{eq:auxbound} and using Lemma~\ref{lem:auxlemma}
yields the inequality
\begin{align}
  \nonumber
  \sum_{t=\tau + 1}^{T + \tau} \inner{x_{t-\tau} - x^*}{g_{t-\tau}}
  \leq &
  \sum_{t=\tau + 1}^{T + \tau} \frac{1}{2} \eta_t \nbr{g_{t-\tau}}^2 +
  \frac{D(x^*\|x_{t}) - D(x^*\|x_{t+1})}{\eta_t} +
  \!\!\!\!\!\!\!\!\sum_{j=1}^{\min(\tau,t-(\tau+1))} \!\!\!\!\!\!\!\! \eta_{t-j} \inner{g_{t-\tau-j}}{g_{t-\tau}}\\
\nonumber
  =&
  \sum_{t=\tau + 1}^{T + \tau} \sbr{\frac{1}{2} \eta_t
    \nbr{g_{t-\tau}}^2  +
  \sum_{j=1}^{\min(\tau,t-(\tau+1))} \eta_{t-j} \inner{g_{t-\tau-j}}{g_{t-\tau}}}\\
  &+\frac{D(x^*\|x_{\tau+1})}{\eta_{\tau+1}}-\frac{
    D(x^*\|x_{T+\tau+1})}{\eta_{T+\tau}} +
  \sum_{t=\tau + 2}^{T + \tau} \sbr{D(x^*\|x_{t})\left (\frac{1}{\eta_t}-\frac{1}{\eta_{t-1}}\right )
}
  \label{eq:theprototype}
\end{align}
%
%
By the Lipschitz property of gradients and the definition of $\eta_t$
we can bound the first summand of the above risk inequality via
\begin{align}
  \label{eq:step1}
  \sum_{t=\tau + 1}^{T + \tau} \frac{1}{2} \eta_t \nbr{g_{t-\tau}}^2
  \leq \sum_{t=\tau + 1}^{T + \tau} \frac{1}{2} \eta_t L^2 =
  \sum_{t=1}^{T} \frac{1}{2} \frac{\sigma}{\sqrt{t}} L^2
  \leq \sigma L^2 \sqrt{T}.
\end{align}
%
%
Next we tackle the terms dependent on $D$. By the assumption on the diameter,
$D(x^*\| x_t) \leq F^2$ for all $x_t$. This yields
\begin{align}
  \frac{D(x^* \| x_{\tau + 1})}{\eta_{\tau+1}} +
  \sum_{t=\tau+2}^{T+\tau} D(x^*\|x_{t}) \left [\frac{1}{\eta_t} -
    \frac{1}{\eta_{t-1}}\right ]
  & \leq \frac{F^2}{\sigma} +
  \frac{F^2}{\sigma} \sum_{t=\tau+2}^{T + \tau}
  \left [\sqrt{t-\tau} - \sqrt{t-\tau-1}\right ]   \nonumber \\
  \label{eq:step2}
  &= \frac{F^2}{\sigma} \sbr{1 + \sqrt{T} -1} =
  \frac{F^2}{\sigma} \sqrt{T}
\end{align}
Here the second to last equality follows from the fact that we have a
telescoping sum. Note that we can discard the contribution of
$-\frac{D(x^*\|x_{T+\tau+1})}{\eta_{T+\tau}}$ since it is always
negative, hence the bound can only become tighter.
%
%

Finally, we address the contribution of the inner products between gradients.
By the Lipschitz property of the gradients we know that
$\inner{g_{t-\tau-j}}{g_{t-\tau}} \leq L^2$. Moreover, $\eta_t$
is monotonically decreasing, hence we can bound the correlation term
in Lemma~\ref{lem:auxlemma} via
\begin{align}
  \label{eq:correlationterm}
\tightsum \eta_{t-j}
\underbrace{\inner{g_{t-j}}{g_{t-\tau}}}_{\leq L^2}
&\leq \min(t-(\tau+1),t)\eta_{\max(\tau,t-(\tau+1))} L^2.
\end{align}
Summing over all contributions yields
\begin{align*}
  \sum_{t=\tau + 1}^{T + \tau} \min(t-(\tau+1),t)\eta_{\max(\tau+1,t-(\tau+1))}
  &= \sum_{t=\tau+1}^{2\tau} (t-(\tau+1)) \eta_{\tau+1} +
  \sum_{t=2\tau+1}^{T + \tau} \tau \eta_{t-\tau} \\
  &= \sum_{t=\tau+1}^{2\tau} (t-(\tau+1)) \frac{\sigma}{\sqrt{1}} +
  \sum_{t=2\tau+1}^{T + \tau} \tau \frac{\sigma}{\sqrt{t - 2\tau}}
  \\
  & \leq \sigma \frac{\tau (\tau-1)}{2} + 2 \sigma \tau \sqrt{T-\tau}
  \leq \frac{\sigma \tau^2}{2} + 2 \sigma \tau \sqrt{T}
\end{align*}
Substituting the bounds for all three terms into the gradient bound
yields
\begin{align}
  R[X] \leq
  \sum_{t=\tau + 1}^{T + \tau} \inner{x_{t-\tau} - x^*}{g_{t-\tau}}
  \leq
  \sigma L^2 \sqrt{T} + F^2 \frac{\sqrt{T}}{\sigma} +
  L^2 \frac{\sigma \tau^2}{2} + 2 L^2 \sigma \tau \sqrt{T}
\end{align}
Plugging in $\sigma = \frac{F}{L \sqrt{2 \tau}}$ changes the RHS to
\begin{align*}
  R[X] \leq \frac{F L \sqrt{T}}{\sqrt{2 \tau}} + F L \sqrt{2 T \tau} + F L
  (\tau/2)^{\frac{3}{2}} + F L \sqrt{2 T \tau} \leq
  F L \sqrt{2 \tau T} \sbr{2 + \frac{1}{2 \tau} + \frac{\tau}{4 \sqrt{T}}}
\end{align*}
Using the fact that $\tau \geq 1$ (otherwise our analysis is vacuous)
and $T \geq \tau^2$ (it is reasonable to assume that we have at least
$O(\tau)$ data per processor) yields the claim.
\end{proof}
In other words the algorithm converges at rate $O(\sqrt{\tau T})$. This is similar to what we would expect in the worst case:
an adversary may reorder instances such as to maximally slow down
progress. In this case a parallel algorithm is no faster than a
sequential code. This result may appear overly pessimistic in practice
but the following example shows that such worst-case scaling behavior
is to be expected:
\begin{lemma}
  \label{lem:optibad}
  Assume that an optimal online algorithm
  with regard to a convex game achieves regret $R[m]$ after seeing
  $m$ instances. Then any algorithm which
  may only use information that is at least $\tau$ instances old has a
  worst case regret bound of $\tau R[m/\tau]$.
\end{lemma}
\begin{proof} The proof is similar to the approach in~\cite{mester05}.
  Our construction works by designing a sequence of functions $f_i$
  where for a fixed $n \in \NN$ all $f_{n \tau + j}$ are identical
  (for $j \in \cbr{1, \ldots, n}$). That is, we send identical
  functions to the algorithm while it has no chance of responding to
  them. Hence, even an algorithm \emph{knowing} that we will see
  $\tau$ identical instances in a row but being disallowed to
  respond to them for $\tau$ instances will do no better than one
  which sees every instance once but is allowed to respond
  instantly. Consequently, the regret incurred will be $\tau$ times
  that of an algorithm seeing $m/\tau$ instances only once each time.
\end{proof}
The useful consequence of Theorem~\ref{th:delay} is that we are
guaranteed to converge \emph{at all} even if we encounter delay (the
latter is not trivial --- after all, we could end up with an
oscillating parameter vector for overly aggressive learning rates).
While such extreme cases hardly occur in practice, we need to make
stronger assumptions in terms of correlation of $f_t$ and the degree
of smoothness in $f_t$ to obtain tighter bounds.

We conclude this section by studying a particularly convenient case: the
setting when the functions $f_i$ are strongly convex with parameter $\lambda>0$ satisfying
\begin{align}
  \label{eq:strongcon}
  f(x^*) \geq f_i(x) + \inner{x^*-x}{\partial_x f(x)} + \frac{\lambda}{2}
  \nbr{x - x^*}^2
\end{align}
Here we can get rid of the $D(x^*\|x_1)$ dependency in the loss bound.
\begin{theorem}
  \label{th:strong}
  Suppose that the functions $f_i$ are strongly convex with parameter $\lambda>0$. Moreover,
  choose the learning rate $\eta_{t} = \frac{1}{\lambda(t-\tau)}$ for $t > \tau$
  and $\eta_t = 0$ for $t \leq \tau$. Then under the assumptions of Theorem~\ref{th:delay} we have the
  following bound:
  \begin{align}
    R[X] \leq \lambda\tau F^2+\sbr{\textstyle \frac{1}{2} + \tau} \frac{L^2}{\lambda}\left (1+\tau+\log T\right )
  \end{align}
\end{theorem}
\begin{proof}
  The proof largely follows that of \cite{BarHazRak08}. The key
  difference is that now we need to take the additional contribution
  of the gradient correlations into account. Using
  \eq{eq:strongcon} we have
   \begin{align*}
     R[X] \leq & \sum_{t=\tau+1}^{T+\tau}
     \inner{x_{t-\tau} - x^*}{g_{t-\tau}} - \frac{\lambda}{2}
     \nbr{x_{t-\tau} - x^*}^2 \\
     \leq & \sum_{t=\tau+1}^{T+\tau}\frac{\eta_t}{2} \nbr{g_{t-\tau}}^2 +
     \hspace{-4mm} \sum_{j=1}^{\min(\tau,t-(\tau+1))} \hspace{-4mm} \eta_{t-j}
     \inner{g_{t-\tau-j}}{g_{t-\tau}} +
    \frac{D(x^*\|x_{t}) - D(x^*\|x_{t+1})}{\eta_t}
    - \frac{\lambda}{2}
     \nbr{x_{t-\tau} - x^*}^2 \\
     \leq & \sum_{t=\tau+1}^{T+\tau}
     \sbr{\frac{1}{2} \eta_{t} + \tau \eta_{\max(t-\tau,\tau+1)}} L^2 +
     \lambda(t-\tau) \sbr{ D(x^*\|x_t) - D(x^*\|x_{t+1})} - \frac{\lambda}{2}
     \nbr{x_{t-\tau} - x^*}^2\\
     \leq & \sum_{t=\tau+1}^{T+\tau}
     \sbr{\frac{1}{2} \eta_{t} + \tau \eta_{\max(t-\tau,\tau+1)}} L^2 +
     \lambda(t-\tau) \sbr{ D(x^*\|x_t) - D(x^*\|x_{t+1})} - \lambda
     D(x^*\|x_{t-\tau})\\
     \leq & \sum_{t=\tau+1}^{T+\tau}
     \sbr{\frac{1}{2} \eta_{t} + \tau \eta_{\max(t-\tau,\tau+1)}} L^2 +
     \lambda(t-(\tau+1)) D(x^*\|x_t) - \lambda(t-\tau)  D(x^*\|x_{t+1})\\
     &+\lambda(D(x^*\|x_t)-
     D(x^*\|x_{t-\tau}))
\end{align*}
Via telescoping:
\begin{align*}
R[X]\leq & ((\tau+1)-(\tau+1))D(x^*\|x_{\tau+1})-\lambda T D(x^*\|x_{T+\tau})\\
&+\sum_{t=1}^{\tau}\lambda(D(x^*\|x_{T+t})-
     D(x^*\|x_{t}))+\sum_{t=\tau+1}^{T+\tau}
\sbr{\frac{1}{2} \eta_{t} + \tau \eta_{\max(t-\tau,\tau+1)}} L^2\\
\leq & \lambda\tau F^2+\sum_{t=\tau+1}^{T+\tau}
\sbr{\frac{1}{2} \eta_{t} + \tau \eta_{\max(t-\tau,\tau+1)}} L^2
\end{align*}

By construction, $\eta_t$ (when $t\geq \tau+1$) is monotonically increasing, hence we have
$\eta_{t} \leq \eta_{\max(t-\tau,\tau+1)}$, so we can:
\begin{align*}
R[X]&\leq \lambda\tau F^2+\sbr{\frac{1}{2} + \tau} L^2\sum_{t=\tau+1}^{T+\tau}\eta_{\max(t-\tau,\tau+1)}\\
\leq & \lambda\tau F^2+\sbr{\frac{1}{2} + \tau} L^2 \sum_{t=1}^{T}\eta_{\max(t,\tau+1)}\\
\leq & \lambda\tau F^2+\sbr{\frac{1}{2} + \tau} L^2 \left (\frac{1}{\lambda} \tau +\sum_{t=1}^{T-\tau}\frac{1}{\lambda t} \right )\\
\leq & \lambda\tau F^2+\sbr{\frac{1}{2} + \tau} \frac{L^2}{\lambda} \left (\tau +1+\log (T-\tau)\right )
\end{align*}

\end{proof}
As before, we pay a linear price in the delay $\tau$.

\section{Decorrelating Gradients}

To improve our bounds beyond the most pessimistic case we need to
assume that the adversary is not acting in the most hostile fashion
possible. In the following we study the opposite case --- namely that
the adversary is drawing the functions $f_i$ iid from an arbitrary
(but fixed) distribution. The key reason for this requirement is that
we need to control the value of $\inner{g_t}{g_{t'}}$ for adjacent
gradients.

The flavor of the bounds we use will be in terms of the
\emph{expected} regret rather than an actual regret. Conversions from
expected to realized regret are standard. See e.g.\ \cite[Lemma 2]{NesVia00}
for an example of this technique. For this purpose we need to take
expectations of sums of copies of \eq{eq:auxbound} in
Lemma~\ref{lem:auxlemma}. Note that this is feasible since
expectations are linear and whenever products between more than one
term occur, they can be seen as products which are \emph{conditionally
  independent} given past parameters, such as $\inner{g_t}{g_{t'}}$
for $|t - t'| \leq \tau$ (in this case no information about $g_t$ can
be used to infer $g_{t'}$ or vice versa, given that we already know
all the history up to time $\min(t,t')-1$. Our informal argument can
be formalized by using martingale techniques. We omit the latter in
favor of a much more streamlined discussion. Since the argument is
rather repetitive (we will prove a number of different bounds) we will
not discuss issues with conditional expectations any further.

A key quantity in our analysis are bounds on the correlation between
subsequent instances. In some cases we will only be able to obtain
bounds on the \emph{expected} regret rather than the actual
regret. For the reasons pointed out in Lemma~\ref{lem:optibad} this is
an in-principle limitation of the setting.

Our first strategy is to assume that $f_t$ arises from a scalar
function of a linear function class. This leads to bounds which, while
still bearing a linear penalty in $\tau$, make do with considerably
improved constants. The second strategy makes stringent smoothness
assumptions on $f_t$, namely it assumes that the gradients themselves
are Lipschitz continuous. This will lead to guarantees for which the
delay becomes increasingly irrelevant as the algorithm
progresses.


\subsection{Covariance bounds for linear function classes}

Many functions $f_t(x)$ depend on $x$
only via an inner product. They can be expressed as
\begin{align}
  \label{eq:linfun}
  f_t(x) = l(y_t, \inner{z_t}{x})
  \text{ and hence }
  g_t(x) = \nabla f_t(x) = z_t
  {\partial_{\inner{z_t}{x}} l(y_t, \inner{z_t}{x})}
\end{align}
Now assume that $\abr{\partial_{\inner{z_t}{x}} l(y_t,
  \inner{z_t}{x})} \leq \Lambda$ for all $x$ and all $t$. This holds,
e.g.\ in the case of logistic regression, the soft-margin hinge loss,
novelty detection. In all three cases we have $\Lambda = 1$. Robust
loss functions such as Huber's regression score \cite[]{Huber81} also
satisfy \eq{eq:linfun}, although with a different constant (the latter
depends on the level of robustness). For such
problems it is possible to bound the correlation between subsequent
gradients via the following lemma:

\begin{lemma}
  \label{lem:alpha}
  Denote by $(y,z), (y',z') \sim \Pr(y,z)$ random variables which are
  drawn independently of $x, x' \in \Xcal$. In this case
  \begin{align}
    \Eb_{y,z,y',z'}\sbr{\inner{\partial_x l(y,
        \inner{z}{x})}{\partial_x l(y', \inner{z'}{x'})}} \leq
    \Lambda^2 \Bigl\|{\Eb_{z,z'}\sbr{z' z^\top}}\Bigr\|_{\mathrm{Frob}}
    =: L^2 \alpha
  \end{align}
\end{lemma}
\label{lem:alpha}
Here we defined $\alpha$ to be the scaling factor which quantifies by
how much gradients are correlated.
\begin{proof}
  By construction we may bound the inner product for linear function
  classes using the Lipschitz constant $\Lambda$. This yields the
  upper bound
  \begin{align*}
    \Lambda^2 \Eb_{z,z'} \sbr{\abr{\inner{z}{z'}}}
    \leq \Lambda^2 \sbr{\Eb_{z,z'}\sbr{\inner{z}{z'}^2}}^{\frac{1}{2}}
    = \Lambda^2 \nbr{\Eb_{z,z'}\sbr{z' z^\top}}_\mathrm{Frob}
  \end{align*}
  Here the first term follows from Lipschitz continuity and the
  inequality is a consequence of the quadratic function being convex.
\end{proof}
We can apply this decorrelation inequality to the previous two
learning algorithms. Theorem~\ref{th:strong} allows a direct
tightening of the guarantees.
While the \emph{order} of the algorithm has not improved relative to
the worst case setting, we have considerably tighter bounds
nonetheless: for instance, for sparse data such as texts the
correlation terms are
rather small, hence the Frobenius norm of the second moment is small
as the second moment matrix is diagonally dominant. Generally,
$\nbr{\Eb\sbr{z z^\top}}_\mathrm{Frob} \leq L^2$ since
the gradient is maximized by having maximal value of the gradient of
$l(y,\inner{z}{x})$ \emph{and} an instance of $z$ with large norm.
Likewise, we may obtain a tighter version of Theorem~\ref{th:delay}.
\begin{corollary}
  \label{cor:delay-alpha}
  Given $\eta_t=\frac{\sigma}{\sqrt{t-\tau}}$ and the conditions of
  Lemma~\ref{lem:alpha} the regret of the delayed update algorithm is
  bounded by
  \begin{align}
    R[X] \leq \sigma L^2 \sqrt{T} + F^2 \frac{\sqrt{T}}{\sigma} +
    L^2 \alpha \frac{\sigma \tau^2}{2} + 2 L^2 \alpha \sigma \tau \sqrt{T}
\end{align}
and consequently for $\sigma^2 = \frac{F^2}{2 \tau \alpha L^2}$
(assuming that $\tau \alpha \geq 1$) and $T \geq \tau^2$ we obtain the
bound
  \begin{align}
    \label{eq:optsigma-delay-alpha}
    R[X] \leq 4FL \sqrt{\alpha \tau T}
  \end{align}
\end{corollary}
\begin{proof}[sketch only]
  The proof is identical to that of Theorem~\ref{th:delay}, except
  that the terms linear and quadratic in $\tau$ are rescaled by a
  factor of $\alpha$. Substituting the new value for $\sigma$ and
  exploiting $\alpha \tau \geq 1$ proves the claim.
\end{proof}

\subsection{Bounds for smooth gradients}

The key to improving the \emph{rate} rather than the \emph{constant}
with regard to which the bounds depend on $\tau$ is to impose further
smoothness constraints on $f_t$. The rationale is quite simple: we
want to ensure that small changes in $x$ do not lead to large changes
in the gradient. This is precisely what we need in order to show that
a small delay (which amounts to small changes in $x$) will not impact
the update that is carried out to a significant amount. More
specifically we assume that the gradient of $f$ is a
Lipschitz-continuous function. That is,
\begin{align}
  \label{eq:lipgrad}
  \nbr{\nabla f_t(x) - \nabla f_t(x')} \leq
  H \nbr{x - x'}.
\end{align}
Such a constraint effectively rules out piecewise linear loss
functions, such as the hinge loss, structured estimation, or the
novelty detection loss. Nonetheless, since this discontinuity only
occurs on a set of measure $0$ delayed stochastic gradient descent
still works very well on them in practice.
We need an auxiliary lemma which allows us to control the magnitude of
the gradient as a function of the distance from optimality:
%
\begin{lemma}
  \label{lem:valboundgrad}
  Assume that $f$ is convex and moreover that $\partial_x f(x)$ is
  Lipschitz continuous with constant $H$. Finally, denote by $x^*$
  the minimizer of $f$. In this case
  \begin{align}
    \label{eq:grobound}
    \nbr{\partial_x f(x)}^2 \leq 2 H [f(x) - f(x^*)].
  \end{align}
\end{lemma}
\begin{proof}
  The proof decomposes into two parts: we first show that the problem
  can be reduced to a one-dimensional setting and secondly we show
  that the claim holds in the one-dimensional case.

  {\bfseries Part 1:} For a given function $f$ with minimizer $x^*$
  and for an arbitrary starting point $x$ we can simply follow the opposite of the
  gradient field $-\partial_x f(x)$ starting at $x$ to arrive at
  $x^*$.\footnote{This is simply gradient descent and related to the
    Picard-Lindel\"of theorem. Without loss of generality we define
    $x^*$ to be the particular minimizer of $f$ that is reached by
    the going in the opposite direction of the gradient flow, whenever there is ambiguity in the choice.}
  The parametrized curve corresponding to the gradient flow is still
  monotonically decreasing, its directed gradient equals
  $-\nbr{\partial_x f(x)}$ along the curve, and moreover, distances
  between points on the curve are bounded from above by the length of
  the path between them. Hence, \eq{eq:grobound} holds for the now
  one-dimensional restriction of $f$. Note that the derivative along the path is strictly negative
  until the end, what one would expect as one heads to a minimum.

  {\bfseries Part 2:}
  Now assume that $f$ is defined on $\RR$. Without loss of
  generality we set $x = 0$ and let $x^* > 0$. Since the gradient
  cannot vanish any faster than the constraint of \eq{eq:grobound} it
  follows that for all $t \in [0, x^*]$ the gradient is bounded from
  above by
  \begin{align}
    f'(t) \leq \min(0, f'(0)+Ht)
  \end{align}
  Note that by construction, the gradient is strictly negative from 0 (inclusive) to $x^*$ (exclusive), hence
  the upper bound of zero. Define $t^* =
  -f'(0)/H$, such that $f'(0)+Ht^*=0$. Clearly $t^* \in [0, x^*]$ since $t^* > x^*$ would imply
  that $f'(x^*)<0$ and $x^*$ is not the minimizer.
  Integrating the lower bound on $f'(t)$ yields
  \begin{align}
    \nonumber
    f(x^*)-f(0) = \int_{0}^{x^*} f'(t) dt \leq \int_{0}^{x^*} \min(0, f'(0)+H t) dt =
    \int_{0}^{t^*} [H t + f'(0)] dt =-
    \frac{ [f'(0)]^2}{2H}
  \end{align}
  Multiplying both sides by $-2 H$ (and switching the inequality) proves the claim. Note that we did
  \emph{not} require in the second part of the proof that $f$ is
  convex or monotonic. This information was only used in part 1 to
  generate the gradient flow.
\end{proof}
This inequality will become useful to show that as we are approaching
optimality, the expected gradient $\partial_x f^*(x)$ also needs to
vanish. Since $g_t$ is assumed to change smoothly with $x$ this
implies that in expectation $g_t$ will vanish for $x \to x^*$ at a
controlled rate. We now state our main result:

\begin{theorem}
  \label{th:coolthm}
  In addition to the conditions of Theorem~\ref{th:delay} assume that the functions $f_i$ are \mbox{i.i.d.},
  $H \geq \frac{L}{4 F \sqrt{\tau}}$ and that $H$ also
  upper-bounds the change in the gradients as in Lemma~\ref{lem:valboundgrad}. Moreover, assume that we choose a learning rate
  $\eta_t = \frac{\sigma}{\sqrt{t-\tau}}$ with $\sigma =
  \frac{F}{L}$. In this case the risk is bounded by
  \begin{align}
    \Eb[R[X]] \leq \sbr{28.3 F^2 H + \frac{2}{3} FL +
      \frac{4}{3} F^2 H \log T} \tau^2 + \frac{8}{3} F L \sqrt{T}.
  \end{align}
\end{theorem}
\begin{proof}
  Our proof is quite similar to that of Theorem~\ref{th:delay}. The
  key differences are that we may now bound  the \emph{expected}
  change between subsequent gradients in terms of the optimality gap
  itself. In particular:
\begin{align}
\Eb[\sum_{t=1}^T f_t(x_t)]&=\sum_{t=1}^T f^*(x_t).
\end{align}
Moreover, observe that:
\begin{align}
f^*(x^*)=\min_{x} \Eb[f_1(x)]&\geq \Eb[\min_{x} f_1(x)]
\end{align}
Moving the minimum inside the expectation makes it so that we can decide
which point after we know what function is drawn, as opposed to before, which makes the problem
easier and the expected cost lower. Note that $\sum_{t=1}^T f_t$, because it is a sum of random functions with mean $f^*$, is itself a random function with mean $Tf^*$. So the same
reasoning applies:
\begin{align}
T f^*(x^*)&\geq \Eb[\min_{x} \sum_{t=1}^T f_t(x)]
\end{align}
Therefore:
\begin{align}
\Eb[R[X]]=\Eb[\max_{x'}\sum_{t=1}^T f_t(x_t)-f_t(x')]&\geq \Eb[\sum_{t=1}^T \sbr{f^*(x_t)-f^*(x^*)}].
\end{align}
We can extract from the proof of Theorem~\ref{th:delay} that:
\begin{align}
R[X]&\leq \sigma L^2\sqrt{T}+\frac{F^2\sqrt{T}}{\sigma}
+\sum_{t=\tau + 1}^{T + \tau}
  \sum_{j=1}^{\min(\tau,t-(\tau+1))} \!\!\!\!\!\!\!\! \eta_{t-j} \inner{g_{t-\tau-j}}{g_{t-\tau}}\\
\Eb[R[X]]&\leq \sigma L^2\sqrt{T}+\frac{F^2\sqrt{T}}{\sigma}
+\sum_{t=\tau + 1}^{T + \tau}
  \sum_{j=1}^{\min(\tau,t-(\tau+1))} \!\!\!\!\!\!\!\! \eta_{t-j} \Eb[\inner{g_{t-\tau-j}}{g_{t-\tau}}]
\end{align}
  Consider the gradient correlation of \eq{eq:correlationterm} for $t
  > \tau$
  \begin{align}
    C_t := \sum_{j=1}^{\tau} \eta_{t-j}
    \inner{g_{t-\tau-j}}{g_{t-\tau}}.
  \end{align}
  We know that $\nbr{x_t - x_{t'}} \leq L \sum_{j=t}^{t'-1} \eta_j$
  since each gradient is bounded by $L$. By the smoothness constraint
  on the gradients this implies that
  $\nbr{\partial_x \left (f_i(x_t) - f_i(x_{t'})\right )} \leq L H \sum_{j=t}^{t'-1}
  \eta_j$. This means that as $\eta_t \to 0$ the error induced
  by the delayed update become a \emph{second order} effect as the
  algorithm converges. In summary, we may bound $C_t$ as follows:
  \begin{align}
    \nonumber
    C_t = & \sum_{j=1}^{\tau} \eta_{t-j} \inner{\nabla f_{t -
      \tau-j}(x_{t-\tau-j})}{\nabla f_{t - \tau}(x_{t-\tau})}\\
     = & \sum_{j=1}^{\tau} \eta_{t-j} \inner{\nabla f_{t -
      \tau-j}(x_{t-\tau})}{\nabla f_{t - \tau}(x_{t-\tau})} + \\
    \nonumber
  &
  \sum_{j=1}^{\tau} \eta_{t-j} \inner{\nabla f_{t-\tau-j}(x_{t-\tau-j}) - \nabla f_{t -\tau-
      j}(x_{t-\tau})}{\nabla f_{t - \tau}(x_{t-\tau})}  \\
    \nonumber
    \leq &
  \sum_{j=1}^{\tau} \eta_{t-j} \sbr{\inner{\nabla f_{t -\tau-
      j}(x_{t-\tau})}{\nabla f_{t - \tau}(x_{t-\tau})} +
  j \eta_{t - 2\tau} L^2 H}
  \end{align}
  Taking expectations of the upper bound is feasible, since all
  $f_{t-\tau-j}$ and $f_{t-\tau}$ are independent of each other and of
  their argument $x_{t-\tau}$. This yields the upper bound
  \begin{align}
    \Eb[C_t] & \leq \sum_{j=1}^{\tau} \eta_{t-j} \sbr{\nbr{\nabla
        f^*(x_{t-\tau})}^2 +   j \eta_{t - 2\tau} L^2 H} \\
    & \leq 2\eta_{t-\tau} \tau H
      \sbr{f^*(x_{t-\tau}) - f^*(x^*)} +   \eta_{t-2\tau}^2 \tau^2 H L^2
    \label{eq:slowchangebound}
  \end{align}
  The second inequality is obtained by appealing to
  Lemma~\ref{lem:valboundgrad} and by using the fact that the learning
  rate is monotonically decreasing.

  What this means is that once the stepsize of the learning rate is
  small enough, second order effects become essentially
  negligible. The overall reduction in the amount by which the bound
  on the expected regret  $f^*(x_t) - f^*(x^*)$ is reduced is given by
  $\tau H \eta_{t-\tau}$. If we wish to limit this reduction to
  $\frac{1}{4}$ this implies for a learning rate of $\eta_t =
  \sigma/\sqrt{t-\tau}$ that we should use \eq{eq:slowchangebound}
  only for $t \geq t_0 := 3\tau + 64 \sigma^2 \tau^2 H^2 \leq 112
  \sigma^2 \tau^2 H^2$ (the latter bounds holds by assumption on $H$).
  We now bound the part of the risk where the effects of the delay are
  sufficiently small. In analogy to \eq{eq:theprototype} we obtain
  \begin{align}
  \nonumber
\Eb[R[x]]&\leq \sigma L^2\sqrt{T}+\frac{F^2\sqrt{T}}{\sigma}
+\sum_{t=\tau + 1}^{T + \tau}
  \sum_{j=1}^{\min(\tau,t-(\tau+1))} \!\!\!\!\!\!\!\! \eta_{t-j} \Eb[\inner{g_{t-\tau-j}}{g_{t-\tau}}]\\
    \nonumber
\Eb[R[x]]&\leq \sigma L^2\sqrt{T}+\frac{F^2\sqrt{T}}{\sigma}
+\sum_{t=\tau + 1}^{t_0}
  \sum_{j=1}^{\min(\tau,t-(\tau+1))} \!\!\!\!\!\!\!\! \eta_{t-j} \Eb[\inner{g_{t-\tau-j}}{g_{t-\tau}}]+\sum_{t=t_0 + 1}^{T + \tau}
  \Eb[C_t]\\
    \nonumber
\Eb[R[x]]&\leq \sigma L^2\sqrt{T}+\frac{F^2\sqrt{T}}{\sigma}
+L^2 \frac{\sigma \tau^2}{2} + 2 L^2 \sigma \tau \sqrt{t_0}+\sum_{t=t_0 + 1}^{T + \tau}
  \Eb[C_t]
  \end{align}
  All but the last term can be bounded in the same way as in
  Theorem~\ref{th:delay}. The sum over the gradient norms can be
  bounded from above by $\sigma L^2 \sqrt{T}$ as in
  \eq{eq:step1}. Likewise, the sum over the divergences can be bounded
  by $\frac{F^2}{\sigma} \sqrt{T}$ as in \eq{eq:step2}. Lastly, since $2H\tau\eta_{t-\tau}\leq \frac{1}{4}$ when $t\geq t_0$ and $f^*(x_{t-\tau}) - f^*(x^*)\geq 0$, the
  sum over the gradient correlations is bounded as follows
  \begin{align}
    \Eb\sbr{ \sum_{t=t_0 + 1}^{T + \tau} C_t} \leq &
    \sum_{t=t_0 + 1}^{T + \tau} \frac{1}{4} \Eb\sbr{f^*(x_{t-\tau}) - f^*(x^*)} + \underbrace{\sum_{t=t_0 +
      1}^{T + \tau} L^2 H \eta_{t-2\tau}^2 \tau^2}_{\leq
    L^2 \tau^2 \sigma^2 H \log T}\\
    \Eb\sbr{ \sum_{t=t_0 + 1}^{T + \tau} C_t} \leq &
    \sum_{t=1}^{T} \frac{1}{4} \Eb\sbr{f^*(x_{t}) - f^*(x^*)}+
    L^2 \tau^2 \sigma^2 H \log T \\
    \Eb\sbr{ \sum_{t=t_0 + 1}^{T + \tau} C_t} \leq &
    \frac{1}{4}\Eb[R[x]]+
    L^2 \tau^2 \sigma^2 H \log T
  \end{align}
  For bounding the sum over $\eta_t^2$ we used a conversion of the sum
  to an integral and the fact that $t_0 > \tau+1$. 
  This first term equals the expected regret over the
  time period $[t_0,  T+\tau]$. Hence, multiplying the overall regret
  bound by $1/(1 - 1/4) = 4/3$ and combining the sum over $[t_0, T]$
  with Theorem~\ref{th:delay} (which covers the segment $[\tau, t_0]$)
  yields the following guarantee:
  \begin{align}
    \frac{3}{4} \Eb[R[X]] \leq &  \sigma L^2 \sqrt{T} + \frac{F^2}{\sigma} \sqrt{T} +
    \frac{1}{2} L^2 \sigma \tau^2 +  2 L^2
    \sigma \tau \sqrt{t_0} +
    L^2 \tau^2 \sigma^2 H \log T
  \end{align}
  Plugging in $\sigma = F/L$ and $t_0 \leq 112 F^2 \tau^2 H^2 / L^2$,
  using the fact that $\tau \geq 1$, and collecting terms yields
  \begin{align}
    \frac{3}{4} \Eb[R[X]] \leq \sbr{21.2 F^2 H + \frac{1}{2}FL +
      F^2 H \log T} \tau^2 + 2 F L \sqrt{T}.
  \end{align}
  Dividing by $\frac{3}{4}$ proves the claim.
\end{proof}
Note that the convergence bound which is $O(\tau^2 \log T + \sqrt{T})$
is governed by two different regimes. Initially, a delay of $\tau$ can
be quite harmful since subsequent gradients are highly correlated. At
a later stage when optimization becomes increasingly an
\emph{averaging} process a delay of $\tau$ in the updates proves to be
essentially harmless. The key difference to bounds of
Theorem~\ref{th:delay} is that now the \emph{rate} of convergence has
improved dramatically and is essentially as good as in sequential
online learning. Note that $H$ does not influence the
\emph{asymptotic} convergence properties but it significantly affects
the initial convergence properties.

This is exactly what one would
expect: initially while we are far away
from the solution $x^*$ parallelism does not help much in providing us
with guidance to move towards $x^*$. However, after a number of steps
online learning effectively becomes an averaging process for variance
reduction around $x^*$ since the stepsize is sufficiently small. In
this case averaging becomes the dominant force, hence parallelization
does not degrade convergence further. Such a setting is desirable ---
after all, we want to have good convergence for extremely large
amounts of data.

\subsection{Bounds for smooth gradients with strong convexity}

We conclude this section with the tightest of all bounds --- the
setting where the losses are all strongly convex and smooth. This
occurs, for instance, for logistic regression with $\ell_2$
regularization. Such a requirement implies that the objective function
$f^*(x)$ is sandwiched between two quadratic functions, hence it is
not too surprising that we should be able to obtain rates comparable
with what is possible in the minimization of quadratic functions. Also
note that the ratio between upper and lower quadratic bound loosely
corresponds to the condition number of a quadratic function --- the
ratio between the largest and smallest eigenvalue of the matrix
involved in the optimization problem.

The analysis is a combination of the proof techniques described in
Theorem~\ref{th:coolthm} in combination with
Theorem~\ref{th:strong}.
\begin{theorem}
  Under the assumptions of Theorem~\ref{th:strong}, in particular,
  assuming that all functions $f_i$ are \mbox{i.i.d} and strongly convex with constant
  $\lambda$ and corresponding learning rate $\eta_t = \frac{1}{\lambda
    (t - \tau)}$ and provided that the loss satisfies
  \eq{eq:grobound} for some constant $H$ we have the following bound on the expected regret:
  \begin{align}
    \Eb \sbr{R[X]} \leq \frac{10}{9} \sbr{\lambda\tau F^2+\sbr{\frac{1}{2} + \tau} \frac{L^2}{\lambda} [1+\tau+ \log
    (3\tau+(H\tau/\lambda))] + \frac{L^2}{2 \lambda} [1+\log T] +
    \frac{\pi^2 \tau^2 H L^2}{6\lambda^2 }}.
  \end{align}
\end{theorem}
\begin{proof}
  As before, we bound the expected correlation between gradients via
  \begin{align*}
    \Eb[C_t] & \leq
    \frac{2 \tau H}{\lambda (t - 2\tau)} \sbr{f^*(x_{t-\tau}) - f^*(x^*)} +
    \frac{\tau^2 H L^2}{\lambda^2 (t - 3\tau)^2}
    \text{ hence (if $t_0>3\tau$)} \\
    \sum_{t=t_0+1}^{T + \tau} \Eb[C_t] & \leq
    \frac{\pi^2 \tau^2 H L^2}{6\lambda^2 }+\frac{2 \tau H}{\lambda (t_0 - 2\tau + 1)} \sum_{t=t_0 +
      1-\tau}^{T}
    \Eb\sbr{f^*\rbr{x_{t}} - f^*(x^*)}
  \end{align*}
  Here the second inequality follows from the fact that the learning
  rate is decreasing and by the fact that $\sum_{n=1}^{\infty}\frac{1}{n^2}=\frac{\pi^2}{6}$. This allows
  us to combine both a bound governing the behavior until $t_0$ and
  a tightened-up bound once gradient changes are small. We obtain
  \begin{align}
    \sbr{1 - \frac{2 \tau H}{\lambda (t_0 - 2\tau + 1)}} \Eb
    \sbr{R[X]} \leq \lambda\tau F^2 +\sbr{\frac{1}{2} + \tau} \frac{L^2}{\lambda} [1+\tau+\log
    t_0] + \frac{L^2}{2\lambda} [1 + \log T] +
    \frac{\pi^2 \tau^2 H L^2}{6\lambda^2 }
    \nonumber
  \end{align}
  By choosing $t_0
  = 3\tau+(H\tau/\lambda)$ we see that the factor on the LHS is bounded by
  $0.9$. This also simplifies expressions on last term of the RHS and
  it yields the inequality
  \begin{align}
    0.9 \Eb \sbr{R[X]} \leq \lambda\tau F^2+\sbr{\frac{1}{2} + \tau} \frac{L^2}{\lambda} [1+\tau+\log
    (3\tau+(H\tau/\lambda))] + \frac{L^2}{2\lambda} [1 +\log T] +
    \frac{\pi^2 \tau^2 H L^2}{6\lambda^2 }.
  \end{align}
  Dividing by $0.9$ proves the claim.
\end{proof}
As before, this improves the rate of the bound. Instead of a
dependency of the form $O(\tau \log T)$ we now have the dependency
$O(\tau^2 + \log T)$. This is particularly desirable for large
$T$. We are now within a small factor of what a fully sequential
algorithm can achieve. In fact, we could make the constant arbitrary
small for large enough $T$.

\section{Bregman Divergence Analysis}
\label{sec:strong}

We now generalize Algorithm~\ref{alg:delay} to Bregman divergences.
In particular, we use the proof technique of \cite[Section 3.1]{ShaSin07}.
We begin by introducing Bregman divergences and strong convexity.
Denote by $\phi: \Bcal \to \RR$ a convex function. Then the
$\phi$-divergence between $x, x' \in \Bcal$ is defined as
\begin{align}
  \label{eq:dphi}
  D_\phi(x\|x') = \phi(x) - \phi(x') - \inner{x-x'}{\nabla \phi(x')}
\end{align}
Moreover, a convex function $f$ is strongly $\sigma$-convex with
respect to $\phi$ whenever the following inequality holds for all $x,
x' \in \Bcal$:
\begin{align}
  \label{eq:strong}
  f(x) - f(x') - \inner{x - x'}{\nabla f(x')} \geq \sigma D_\phi(x\|x').
\end{align}
Finally, for a convex function $f$ denote by $f^*$ the
Fenchel-Legendre dual of $f$. It is given by $f^*(y) = \sup_{x} \inner{x}{y}
- f(x)$. We are now able to define the implicit update version of
Algorithm~\ref{alg:delay}.
\begin{algorithm}[h]
  \caption{Delayed Stochastic Gradient Descent with Implicit Updates \label{alg:delay-implicit}}
  \begin{algorithmic}
    \STATE {\bfseries Input:} scalar $\sigma > 0$, delay $\tau \in
    \NN$ and convex function $\phi$.
    \STATE Set $x_1 \ldots, x_\tau = 0$ and compute corresponding $g_t = \nabla f_t(x_t)$.
    \FOR{$t = \tau + 1$ {\bfseries to} $T + \tau$}
    \STATE Obtain $f_t$ and incur loss $f_t(x_t)$
    \STATE Compute $g_t := \nabla f_t(x_t)$ and set $\eta_t =
    \frac{\sigma}{\sqrt{t- \tau}}$
    \STATE Update $x_{t+1} = \nabla \phi^*\rbr{\nabla \phi(x_t) - \eta_t g_{t-\tau}}$
    \ENDFOR
  \end{algorithmic}
\end{algorithm}
It is easy to check that Algorithm~\ref{alg:delay} is a special case
of Algorithm~\ref{alg:delay-implicit}. For $\phi(x) = \frac{1}{2} \nbr{x}^2$ we have that
$\phi^* = \phi$ and $\nabla \phi(x) = x$. If $\phi$ is the unnormalized logarithm we obtain delayed
exponential gradient descent. We state the following lemma without
proof, since it is virtually identical to that of \cite{ShaSin07}:

\begin{lemma}
  \label{th:strong-anotherone}
  Assume that $\phi$ is 1-strongly convex with respect to the norm
  associated with $\Bcal$. Then for any $x^* \in \Bcal$, and in
  particular the loss minimizer, the following
  holds
  \begin{align}
    \label{eq:delay-chain}
    \inner{x_t - x^*}{g_{t - \tau}} \leq
    \frac{D_\phi(x^*\|x_t) - D_\phi(x^*\|x_{t+1})}{\eta_t} + \eta_t
    \frac{1}{2} \nbr{g_{t-\tau}}^2_*
  \end{align}
\end{lemma}
\begin{theorem}
  \label{th:delay-bregmann}
  Assume that the implicit updates associated with $\phi$ are
  Lipschitz, that is
  \begin{align}
    \label{eq:lipcon}
    \nbr{\nabla \phi^*(\nabla \phi(x) - x') - x}
    \leq
    \Phi \nbr{x'}
  \end{align}
  for some $\Phi > 0$. Then the delayed update algorithm has a regret
  bound of the form
  \begin{align}
    \label{eq:bound-bregmann}
    R[X] \leq \sigma L^2 \sqrt{T} + F^2 \frac{\sqrt{T}}{\sigma} +
    L^2 \Phi \frac{\sigma \tau^2}{2} + 2 L^2 \Phi \sigma \tau \sqrt{T}
\end{align}
and consequently for $\sigma^2 = \frac{F^2}{2 \tau \Phi L^2}$
(assuming that $\tau \Phi \geq 1$) and $T \geq \tau^2$ we obtain the
bound
  \begin{align}
    \label{eq:optsigma-delay-bregmann}
    R[X] \leq 4FL \sqrt{\Phi \tau T}
  \end{align}
\end{theorem}
\begin{proof}
  To apply the regret bounds we need to replace $\inner{x_t -
    x^*}{g_{t - \tau}}$ in \eq{eq:delay-chain} by a term which uses
  $x_{t - \tau}$ instead of $x_t$. This can be achieved by telescoping
  via
  \begin{align}
    \inner{x_t - x^*}{g_{t - \tau}} =
    \inner{x_{t-\tau} - x^*}{g_{t - \tau}}  + \sum_{j=0}^{\tau - 1}
    \inner{x_{t - j} - x_{t-j-1}}{g_{t - \tau}}
  \end{align}
  The key difference to before is that now the difference between
  subsequent weight vectors does \emph{not} constitute the gradient
  anymore. To obtain the same type of bounds that yielded
  Theorem~\ref{th:delay} we exploit continuity in the forward and
  reverse transform via \eq{eq:lipcon}.
  This yields $\inner{x_t - x^*}{g_{t - \tau}} \geq \inner{x_{t-\tau} -
    x^*}{g_{t - \tau}}  - \tau \eta_{t - \tau} \Phi L^2$.
  Plugging this bound into a sum over $T$ terms and using the argument
  as in Theorem~\ref{th:delay} proves the claim.
\end{proof}
Obtaining bounds that are as tight as Theorem~\ref{th:coolthm} is
subject of further work. We anticipate, however, that this may not be
quite as easy, in particular whenever functions can change
significantly after just seeing a small number of examples, as is the
case for exponentiated gradient descent. Here a delay can be
considerably more harmful than in the simple stochastic gradient
descent scenario.

\section{Experiments}

\begin{figure}[htb]
\begin{tabular}{cc}
\includegraphics[width=0.45\textwidth]{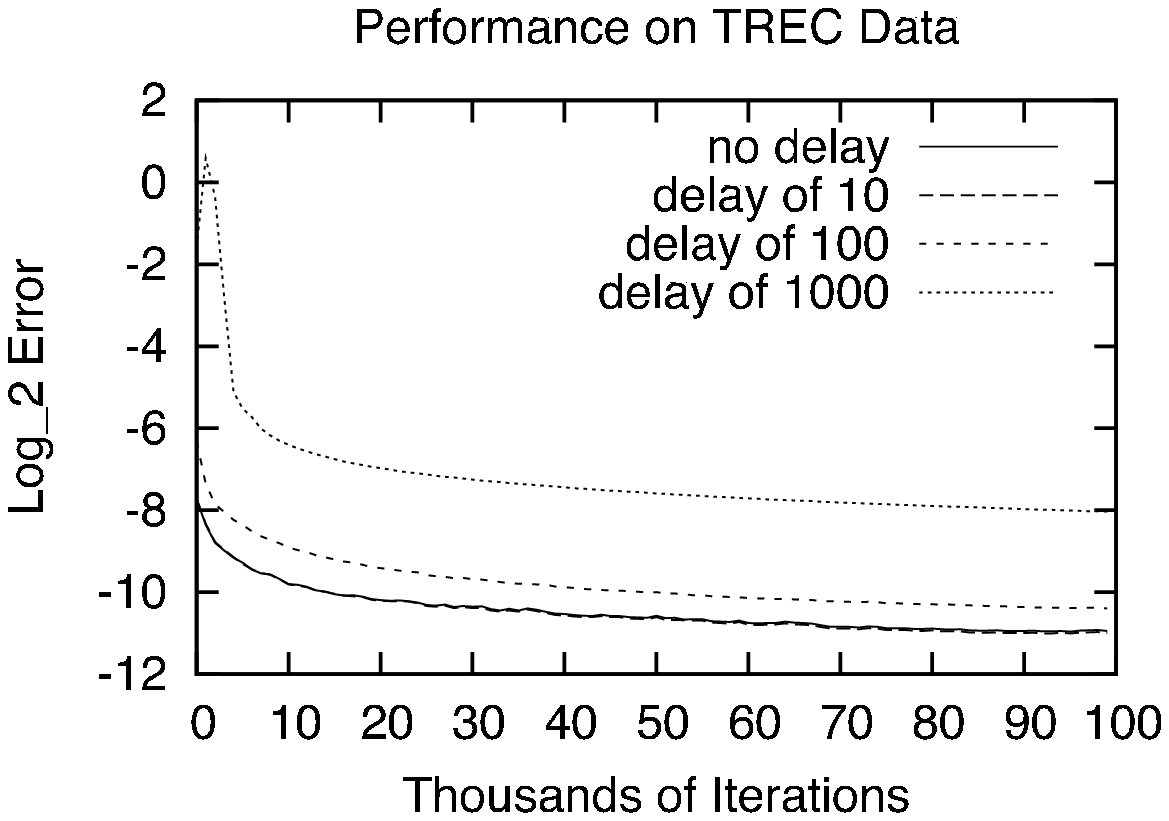}
&
\includegraphics[width=0.45\textwidth]{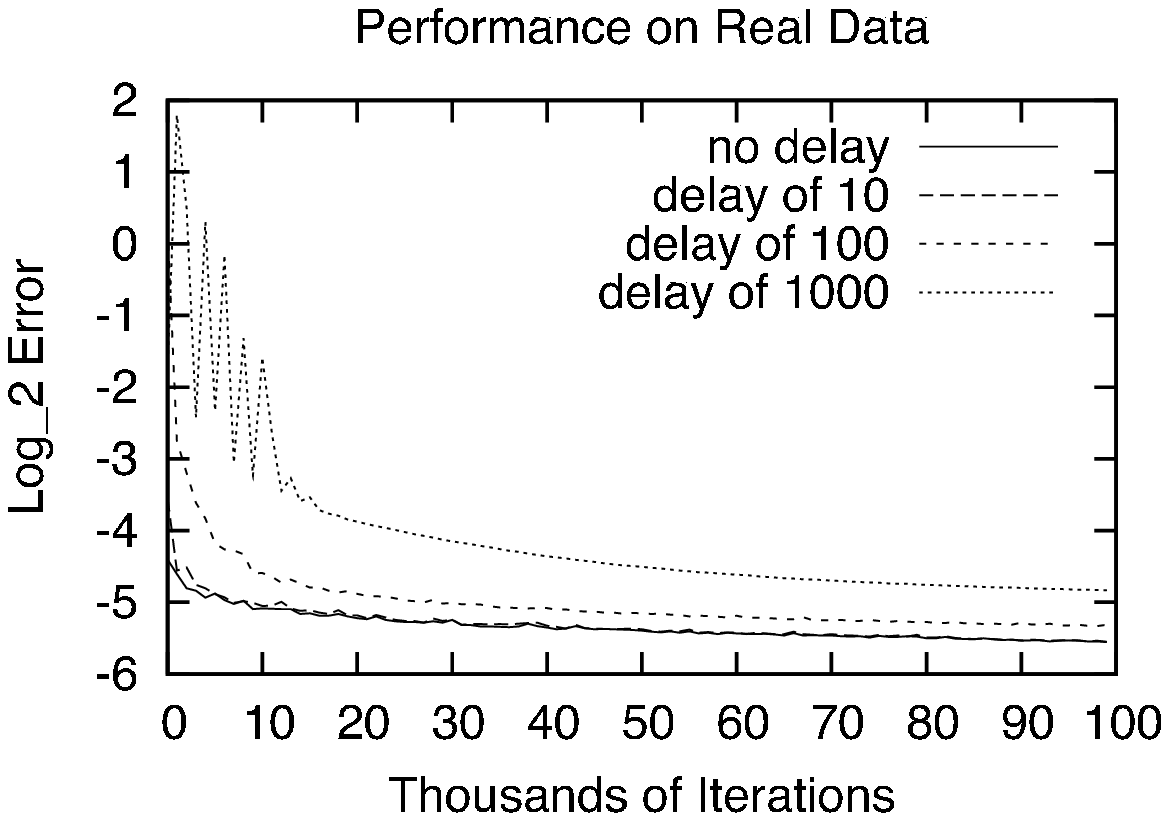}\\
a & b \\
\includegraphics[width=0.45\textwidth]{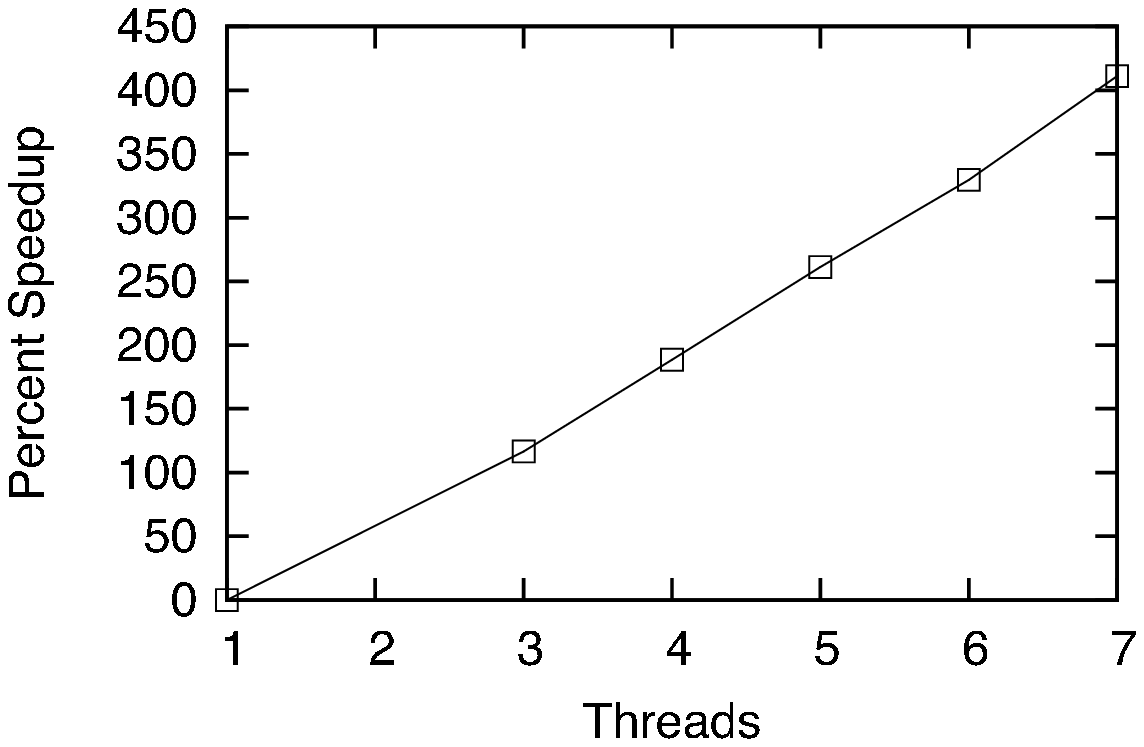}\\
c &
\end{tabular}
\caption{(a) Experiments with simulated delay on the TREC dataset (b) Experiments with simulated delay on the (harder) proprietary dataset (c) Time performance on a subset of the TREC dataset which fits into memory, using the quadratic representation. There was either one thread (a serial implementation) or 3 or more threads (master and 2 or more slaves). \label{fig:convergence}}
\end{figure}

In our experiments we focused on pipelined optimization. In
particular, we used two different training sets that were based on
e-mails: the TREC dataset~\cite[]{Cormack07}, consisting of 75,419
e-mail messages, and a proprietary (significantly harder) dataset of
which we took 100,000 e-mails. These e-mails were tokenized by
whitespace. The problem there is one of binary classification, that is
we are interested in minimizing
\begin{align}
  \label{eq:scalaloss}
  f_t(x) = l(y_t \inner{z_t}{x})
  \text{ where }
  l(\chi) =
  \begin{cases}
    \frac{1}{2} - \chi & \text{ if } \chi \leq 0 \\
    \frac{1}{2} (\chi - 1)^2 & \text{ if } \chi \in [0, 1] \\
    0 & \text{ otherwise}
  \end{cases}
\end{align}
Here $y_t \in \cbr{\pm 1}$ denote the labels of the binary
classification problem, and $l$ is the smoothed quadratic soft-margin
loss of \cite{LanLiStr07}. We used two feature representations: a
linear one which amounted to a simple bag of words representation, and
a quadratic one which amounted to generating a bag of word pairs
(consecutive or not).

To deal with high-dimensional feature spaces we used hashing
\cite[]{Weinbergeretal09}. In particular, for the TREC dataset we used
$2^{18}$ feature bins and for the proprietary dataset we used $2^{24}$
bins. Note that hashing comes with performance guarantees which state
that the canonical distortion due to hashing is sufficiently small for
the dimensionality we picked.

We tried to address the following issues in our simulation:
\begin{enumerate}
\item The obvious question is a systematic one: how much of a
  convergence penalty do we incur in practice due to delay. This
  experiment checks the goodness of our bounds. We checked convergence
  for a system where
  the delay is given by $\tau\in \{0,10,100,1000\}$.
\item Secondly, we checked on an actual parallel implementation
  whether the algorithm scales well. Unlike the previous check
  includes issues such as memory contention, thread synchronization,
  and general feasibility of a delayed updating architecture.
\end{enumerate}

\paragraph{Implementation}

The code was written in Java, although several of the fundamentals
were based upon VW~\cite[]{LanLiStr07}, that is, hashing and the choice
of loss function. We added regularization using lazy updates of the
parameter vector (i.e.\ we rescale the updates and occasionally
rescale the parameter). This is akin to Leon Bottou's SGD code. For robustness, we used $\eta_{t}=\frac{1}{\sqrt{t}}$.

All timed experiments were run on a single, 8 core machine with 32 GB
of memory. In general, at least 6 of the cores were free at any
given time.  In order to achieve advantages of
parallelization, we divide the feature space $\{1\ldots n\}$ into
roughly equal pieces, and assign a slave thread to each piece. Each
slave is given both the weights for its pieces, as well as the
corresponding pieces of the examples. The master is given the label of
each example. We compute the dot product separately on each piece, and
then send these results to a master. The master adds the pieces
together, calculates the update, and then sends that back to the
slaves. Then, the slaves update their weight vectors in proportion to
the magnitude of the central classifier. What makes this work quickly
is that there are multiple examples in flight through this dataflow
simultaneously. Note that between the time when a dot product is
calculated for an
example and when the results have been transcribed, the weight vector has
been updated with several other earlier examples and the dot products
have been calculated from several later examples. As a safeguard we
limited the maximum delay to 100 examples. In this case the compute slave
would simply wait for the pipeline to clear.

The first experiment that we ran was a simulation where we artificially added a delay between the update and the product (Figure~\ref{fig:convergence}a). We ran this experiment using linear features, and observed that the performance did not noticeably degrade with a delay of 10 examples, did not significantly degrade with a delay of 100, but with a delay of 1000, the performance became much worse.

The second experiment that we ran was with a proprietary dataset (Figure~\ref{fig:convergence}b). In this case, the delays hurt less; we conjecture that this was because the information gained from each example was smaller. In fact, even a delay of 1000 does not result in particularly bad performance.

Encouraged by these results, we tried to parallelize these exact experiments (results not shown). This turned out to be impossible: a serial implementation alone handled over 150,000 examples/second. However, when you consider more complex problems, such as with a quadratic representation, then a single example takes slightly above one millisecond. In this domain, we found that parallelization dramatically improved performance (Figure~\ref{fig:convergence}c). In this case, we loaded a small number of examples that could fit into memory,\footnote{ideally, one could design code optimized for quadratic representations, and never explicitly generate the whole example} and showed that the parallelization improved speed dramatically.

\section{Summary and Discussion}

Trying the type of delayed updates presented here is a natural approach to the problem: however, intuitively, having a delay of $\tau$ is like having a learning rate that is $\tau$ times larger. In this paper, we have shown theoretically how independence between examples can make the actual effect much smaller.

The experimental results showed three important aspects: first of all, small simulated delayed updates do not hurt much, and in harder problems they hurt less; secondly, in practice it is hard to speed up ``easy'' problems with a small amount of computation, such as e-mails with linear features; finally, when examples are larger or harder, the speedups can be quite dramatic.

\newpage

\bibliography{bibfile}

\end{document}